\documentclass[brochure,english,12pt]{bourbaki}
\usepackage[matrix,arrow]{xy}
\usepackage{amssymb,amsfonts,amsmath,footnote}
\usepackage[francais]{babel}
\date{Juin 2016}
\bbkannee{68\`eme ann\'ee, 2015-2016}
\bbknumero{1119}
\title{The Liouville function in short intervals}
\subtitle{after Matom{\" a}ki and Radziwi{\l \l}}
\author{Kannan SOUNDARARAJAN}
\address{Stanford University\\
Department of Mathematics\\
450 Serra Mall, Building 380\\ 
Stanford, CA 94305-2125, U.S.A.}
\email{ksound@math.stanford.edu}

\begin{document}
  
  \maketitle
  \noindent {\it R\'esum\'e\/}:\footnote{Je sais gr\'e au Prof.\ Tokieda d'avoir bien voulu traduire ce r\'esum\'e en 
fran{\c c}ais.}\footnote{Je remercie  Tokieda d'avoir traduit la note ci-dessus en fran{\c c}ais.} 
\footnote{Je remercie Tokieda d'avoir traduit la note ci-dessus en fran{\c c}ais.}
    La fonction de Liouville $\lambda$ est une fonction compl\`etement multiplicative \`a valeur
$\lambda(n) = +1$ [resp.\ $-1$] si $n$ admet un nombre pair [resp.\ impair] de facteurs premiers,
compt\'es avec multiplicit\'e.  On s'attend \`a ce qu'elle se comporte comme une collection  \guillemotleft al\'eatoire\guillemotright  \ 
 de signes, $+1$ et $-1$ \'etant \'equiprobables.  Par exemple, une conjecture c\'el\`ebre de Chowla dit que les
valeurs $\lambda(n)$ et $\lambda(n+1)$ (plus g\'en\'eralement en arguments translat\'es par
$k$ entiers distincts fixes) ont corr\'elation nulle.  Selon une autre croyance r\'epandue,
presque tous les intervalles de longueur divergeant vers l'infini devraient donner \`a peu pr\`es le
m\^eme nombre de valeurs $+1$ et $-1$ de $\lambda$.  R\'ecemment Matom\"aki et
Radziwi\l\l\ ont \'etabli que cette croyance \'etait en effet vraie, et de plus \'etabli une variante d'un 
tel r\'esultat pour une classe g\'en\'erale de fonctions multiplicatives.  Leur collaboration ult\'erieure 
avec Tao a conduit ensuite \`a la d\'emonstration des versions moyennis\'ees de la conjecture de 
Chowla, ainsi qu'\`a celle de l'existence de nouveaux comportements de signes de la fonction
de Liouville.  Enfin un dernier travail de Tao v\'erifie une version logarithmique de ladite conjecture et,
de l\`a, r\'esout la conjecture de la discr\'epance d'Erd{\H o}s.  Dans ce S\'eminaire je vais exposer
quelques-unes des id\'ees ma{\^\i}tresses sous-jacentes au travail de Matom\"aki et
Radziwi\l\l.

 \section{Introduction}

The Liouville function $\lambda$ is defined by setting $\lambda(n) =1$ if 
$n$ is composed of an even number of prime factors (counted with multiplicity) and $-1$ if $n$ is composed 
of an odd number of prime factors.  Thus, it is a completely multiplicative function taking the value $-1$ at all primes $p$.  
The Liouville function is closely related to the M{\" o}bius function $\mu$, which equals $\lambda$ on square-free integers, and which equals $0$ on integers that are divisible by the square of a prime.  

The Liouville function takes the 
values $1$ and $-1$ with about equal frequency: as $x\to \infty$ 
\begin{equation} 
\label{1.1} 
\sum_{n\le x} \lambda(n)  = o(x), 
\end{equation}
and this statement (or the closely related estimate $\sum_{n\le x} \mu(n) = o(x)$) is equivalent to the prime number theorem.   Much more is expected to be true, and the sequence of values of $\lambda(n)$ should appear more or less like a random sequence of $\pm 1$.   For example, one expects that the sum in \eqref{1.1} has ``square-root cancelation":  for any $\epsilon >0$
\begin{equation} 
\label{1.2} 
\sum_{n\le x} \lambda(n) = O(x^{\frac 12 +\epsilon}), 
\end{equation} 
and this bound is equivalent to the Riemann Hypothesis (for a more precise version of this equivalence see \cite{S}).  In particular, the Riemann Hypothesis implies that 
\begin{equation} 
\label{1.3} 
\sum_{x< n \le x+ h} \lambda(n) = o(h), \qquad \text{provided  } h >x^{\frac 12 +\epsilon},  
\end{equation} 
and a refinement of this, due to Maier and Montgomery \cite{MM}, permits the range $h > x^{1/2}(\log x)^A$ for a suitable constant $A$.
Unconditionally, Motohashi \cite{Mo} and Ramachandra \cite{Ra} showed independently that 
\begin{equation} 
\label{1.4} 
\sum_{x < n \le x+h} \lambda(n) = o(h), \qquad \text{provided  } h> x^{\frac 7{12}+\epsilon}.  
\end{equation} 
The analogy with random $\pm 1$ sequences would suggest cancelation in every short interval 
as soon as $h> x^{\epsilon}$ (perhaps even $h \ge (\log x)^{1+\delta}$ is sufficient).  Instead of asking for 
cancelation in every short interval, if we are content with results that hold for almost all short intervals, then 
more is known.  Assuming the Riemann Hypothesis, Gao \cite{G} established that if $h \ge (\log X)^{A}$ for a suitable (large) 
constant $A$, then 
\begin{equation} 
\label{1.5} 
\int_X^{2X} \Big| \sum_{x<n\le x+h} \lambda(n) \Big|^2 dx = o(X h^2), 
\end{equation} 
so that almost all intervals $[x,x+h]$ with $X\le x\le 2X$ exhibit cancelation in the values of $\lambda(n)$.  
Unconditionally 
one can use zero density results to show that almost all intervals have substantial cancelation if $h> X^{1/6+\epsilon}$. 
To be precise, Gao's result (as well as the results in \cite{S}, \cite{MM}, \cite{Mo}, \cite{Ra}) was established 
for the M{\" o}bius function, but only minor changes are needed to cover the Liouville function.

The results described above closely parallel results about the distribution of prime numbers.  We have already 
mentioned that \eqref{1.1} is equivalent to the prime number theorem:
\begin{equation} 
\label{1.6} 
\psi(x) = \sum_{n\le x} \Lambda(n) = x+o(x), 
\end{equation} 
where $\Lambda(n)$, the von Mangoldt function, equals $\log p$ if $n>1$ is a power of the prime $p$, and $0$ 
otherwise.   Similarly, in analogy with \eqref{1.2}, a classical equivalent formulation of the Riemann Hypothesis states that 
\begin{equation} 
\label{1.7} 
\psi(x) = x+ O\big(x^{\frac 12} (\log x)^2 \big),
\end{equation} 
so that a more precise version of \eqref{1.3} holds 
\begin{equation} 
\label{1.8} 
\psi(x+h) - \psi(x)= \sum_{x < n\le x+h} \Lambda(n)  = h+ o(h), \qquad \text{provided  } h> x^{1/2}(\log x)^{2+\epsilon}. 
\end{equation} 
Analogously to \eqref{1.4}, Huxley \cite{H} (building on a number of previous results) showed unconditionally that 
  \begin{equation} 
  \label{1.9} 
  \sum_{x<n\le x+h} \Lambda(n) \sim h, \qquad \text{provided  } h> x^{\frac 7{12}+\epsilon}. 
  \end{equation} 
 Finally, Selberg \cite{Se} established that if the Riemann hypothesis holds, and  $h \ge (\log X)^{2+\epsilon}$ then 
 \begin{equation} 
 \label{1.10} 
 \int_X^{2X} \Big| \sum_{x<n\le x+h} \Lambda(n) - h \Big|^2 dx = o(Xh^2), 
 \end{equation} 
 so that almost all such short intervals contain the right number of primes.  Unconditionally one can use Huxley's zero density 
 estimates to show that almost all intervals of length $h>X^{1/6+\epsilon}$ contain the right 
 number of primes. 
 
The results on primes invariably preceded their analogues for the Liouville (or M{\" o}bius) function, and 
often there were some extra complications in the latter case.  For example, the work of Gao is much more 
involved than Selberg's estimate \eqref{1.10}, and the corresponding range in \eqref{1.5} is a little weaker.  
Although there has been dramatic recent progress in sieve theory and understanding gaps between primes, the estimates \eqref{1.8}, \eqref{1.9} and \eqref{1.10} have not been substantially improved for a long time.  So it came as a great surprise when Matom{\" a}ki and Radziwi{\l }{\l} established that the Liouville function exhibits cancelation in almost all short intervals, as soon as the length of the interval tends to infinity --- that is, obtaining qualitatively a definitive version of \eqref{1.5} unconditionally!

 \begin{theo} [Matom{\" a}ki and Radziwi{\l}{\l} \cite{MR2}] \label{thm1}  For any $\epsilon >0$ there exists $H(\epsilon)$ 
 such that for all 
 $H(\epsilon) < h\le X$ we have 
 $$ 
 \int_X^{2X} \Big| \sum_{x<n\le x+h} \lambda(n) \Big|^2 dx  \le \epsilon Xh^2.
 $$ 
 Consequently, for $H(\epsilon) < h\le X$ one has 
 $$ 
\Big| \sum_{x < n\le x+h} \lambda(n) \Big| \le \epsilon^{\frac 13}h, 
$$
except for at most $\epsilon^{\frac 13} X$ integers $x$ between $X$ and $2X$.  
\end{theo}

 As mentioned earlier, the sequence $\lambda(n)$ is expected to resemble a random $\pm 1$ sequence, and 
 the expected square-root cancelation in the interval $[1,x]$ and cancelation in short intervals $[x,x+h]$ reflect the corresponding cancelations in random $\pm 1$ sequences.  Another natural way to capture the apparent randomness of $\lambda(n)$ is to fix a pattern of consecutive signs $\epsilon_1$, $\ldots$, $\epsilon_k$ (each $\epsilon_j$ being $\pm 1$) and ask for the number of $n$ such that $\lambda(n+j) = \epsilon_j$ for each $1\le j\le k$.  
 If the Liouville function behaved randomly, then one would expect that the density of $n$ with this sign pattern should be $1/2^k$.  
 
 \begin{conj} [Chowla \cite{C}]  Let $k\ge 1$ be an integer, and let $\epsilon_j = \pm 1$ for $1\le j\le k$.  Then as $N \to \infty$ 
 \begin{equation} 
 \label{C1.1} 
 |\{ n\le N: \lambda(n+j) = \epsilon_j \text{  for all  } 1\le j\le k\}|  = \Big( \frac{1}{2^k} + o(1) \Big) N. 
 \end{equation} 
 Moreover, if $h_1$, $\ldots$, $h_k$ are any $k$ distinct integers then, as $N \to \infty$,  
 \begin{equation} 
 \label{C1.2} 
 \sum_{n\le N} \lambda(n+h_1)\lambda(n+h_2)\cdots \lambda(n+h_k) = o(N). 
 \end{equation} 
 \end{conj} 

Observe that $\prod_{j=1}^{k} (1+\epsilon_j \lambda(n+j)) =2^k$ if $\lambda(n+j) =\epsilon_j$, and $0$ otherwise. 
Expanding this product out, and summing over $n$, it follows that \eqref{C1.2} implies \eqref{C1.1}.   It is also clear that 
\eqref{C1.2} follows if \eqref{C1.1} holds for all $k$.  

The Chowla conjectures resemble the Hardy-Littlewood conjectures on prime $k$-tuples, and little is known in their direction.   The prime number theorem, in its equivalent form \eqref{1.1}, shows that $\lambda(n)=1$ and $-1$ about equally often, so 
that \eqref{C1.1} holds for $k=1$.   When $k=2$, there are four possible patterns of signs for $\lambda(n+1)$ and $\lambda(n+2)$, and as a consequence of Theorem \ref{thm1} 
it follows that each of these patterns appears a positive proportion of the time.   
 For $k=3$, Hildebrand \cite{Hi} was able to show that all eight patterns of three consecutive signs occur infinitely often.   By combining Hildebrand's ideas with the work in \cite{MR2}, Matom{\" a}ki, Radziwi{\l}{\l}, and Tao \cite{MRT2} have shown that all eight patterns appear a positive proportion of the time.   It is still unknown whether all sixteen four term patterns of signs appear infinitely often (see \cite{BE} for some related work).

\begin{theo}[Matom{\" a}ki, Radziwi{\l}{\l}, and Tao \cite{MRT2}] \label{thm2} For any of the eight choices of $\epsilon_1$, $\epsilon_2$, $\epsilon_3$ all $\pm 1$ we have 
$$ 
\liminf_{N \to \infty} \frac{1}{N} |\{ n\le N: \lambda(n+j)=\epsilon_j, \ \ j=1, 2, 3 \}| > 0.  
$$ 
\end{theo}

We turn now to \eqref{C1.2}, which is currently open even in the simplest case of showing $\sum_{n\le N} \lambda(n) \lambda(n+1) = o(N)$.  By refining 
the ideas in \cite{MR2}, Matom{\" a}ki, Radziwi{\l \l} and Tao \cite{MRT1} showed that a version of Chowla's conjecture \eqref{C1.2} holds  if we permit a 
small averaging over the parameters $h_1$, $\ldots$, $h_k$.   
 
\begin{theo}[Matom{\" a}ki, Radziwi{\l}{\l}, and Tao \cite{MRT1}] \label{thm3}  Let $k$ be a natural number, and let $\epsilon > 0$ be given. 
There exists $h(\epsilon,k)$ such that for all $x\ge h \ge h(\epsilon,k)$ we have 
$$ 
\sum_{1\le h_1, \ldots, h_k \le h } \Big| \sum_{n\le x}  \lambda(n+h_1) \cdots \lambda(n+h_k) \Big| \le \epsilon h^k x. 
$$ 
\end{theo} 

Building on the ideas in \cite{MR2} and \cite{MRT2}, and introducing further new ideas, Tao \cite{T1} has established a 
logarithmic version of Chowla's conjecture \eqref{C1.2} in the case $k=2$.   A lovely and easily stated consequence of Tao's work is 
\begin{equation} 
\label{Tao1} 
\sum_{n\le x} \frac{\lambda(n)\lambda(n+1)}{n}  = o(\log x). 
\end{equation}  
Results such as \eqref{Tao1}, together with their extensions to general multiplicative functions bounded by $1$, form a crucial 
part of Tao's remarkable resolution of the Erd{\H o}s discrepancy problem \cite{T2}:  If $f$ is any function from the positive integers to $\{-1, +1\}$ then 
$$ 
\sup_{d, n } \Big | \sum_{j=1}^{n} f(jd) \Big| = \infty. 
$$

While we have so far confined ourselves to the Liouville function, the work of Matom{\" a}ki and Radziwi{\l \l} applies more 
broadly to general classes of multiplicative functions.   For example, Theorem \ref{thm1} holds in the following more general form.  
Let $f$ be a multiplicative function with $-1\le f(n) \le 1$ for all $n$.  For any $\epsilon >0$ there exists $H(\epsilon)$ such that 
if $H(\epsilon) < h\le X$ then 
\begin{equation} 
\label{Mult1} 
\Big| \sum_{x<n\le x+h} f(n) - \frac{h}{X} \sum_{X \le n \le 2X} f(n) \Big| \le \epsilon h, 
\end{equation} 
for all but $\epsilon X$ integers $x$ between $X$ and $2X$.  In other words, for almost all intervals of length $h$, the local average of $f$ in 
the short interval $[x,x+h]$ is close to the global average of $f$ between $X$ and $2X$.  We should point out that 
this result holds uniformly for all multiplicative functions $f$ with $-1\le f(n)\le 1$ --- that is, the quantity $H(\epsilon)$ depends only on $\epsilon$ and is independent of $f$.  
A still more general formulation (needed for Theorem \ref{thm3}) 
may be found in Appendix 1 of \cite{MRT1}.

The work of Matom{\" a}ki and Radziwi{\l \l} permits a number of elegant corollaries, and we highlight two such results; see Section 8 for a brief discussion of their proofs.  



\begin{coro} [Matom{\" a}ki and Radziwi{\l \l} \cite{MR2}] \label{cor2}  For every $\epsilon>0$, there exists a constant $C(\epsilon)$ such that 
for all large $N$, the interval $[N,N+C(\epsilon) \sqrt{N}]$ contains an integer all of whose prime factors are below $N^{\epsilon}$. 
\end{coro}  

Integers without large prime factors (called {\sl smooth} or {\sl friable} integers) have been extensively studied,  and the existence of 
smooth numbers in short intervals is of interest in understanding the complexity of factoring algorithms.  Previously Corollary \ref{cor2} 
was only known conditionally on the Riemann hypothesis (see \cite{So2}).  Further, \eqref{Mult1} shows that almost all intervals with length 
tending to infinity contain the right density of smooth numbers (see Corollary 6 of \cite{MR2}).

\begin{coro}[Matom{\" a}ki and Radziwi{\l\l} \cite{MR2}]  \label{cor3} Let $f$ be a real valued multiplicative function such that (i) $f(p)<0$ for 
some prime $p$, and (ii) $f(n)\neq 0$ for a positive proportion of integers $n$.   Then for all large $N$ the non-zero values of $f(n)$ with 
$n\le N$ exhibit a positive proportion of sign changes: precisely, for some $\delta >0$ and all large $N$, there are $K \ge \delta N$ integers $1\le n_1 < n_2 < \ldots < n_K \le N$ 
such that $f(n_j) f(n_{j+1}) <0$ for all $1\le j\le K-1$.  
\end{coro} 

The conditions (i) and (ii) imposed in Corollary \ref{cor3} are plainly necessary for $f$ to have a positive proportion of 
sign changes.  For the Liouville function, which is never zero, Corollary \ref{cor3} says that $\lambda(n) = -\lambda(n+1)$ 
for a positive proportion of values $n$; of course this fact is also a special case of Theorem \ref{thm2}.   Even for the M{\" o}bius function, 
Corollary \ref{cor3} is new, and improves upon the earlier work of Harman, Pintz and Wolke \cite{HPW}; for general multiplicative functions, it improves upon the 
earlier work of Hildebrand \cite{Hi2} and Croot \cite{Cr}.   Corollary \ref{cor3} also applies to the Hecke eigenvalues of holomorphic newforms, where 
Matom{\" a}ki and Radziwi{\l \l} \cite{MR1} had recently established such a result by different means.   The sign changes of Hecke eigenvalues are related 
to the location of ``real zeros" of the newform $f(z)$ (see \cite{GhSa}), and this link formed the initial impetus for the work of Matom{\" a}ki and Radziwi{\l \l}.

The rest of this article will give a sketch of some of the ideas behind Theorem \ref{thm1}; the reader may also find it useful to consult \cite{MR3, T3}.   
For ease of exposition, in our description of 
the results of Matom{\" a}ki and Radziwi{\l \l} we have chosen to give a qualitative sense of their work.  In fact Matom{\" a}ki and Radziwi{\l \l} establish  
Theorem \ref{thm1} in the stronger quantitative form (for any $2\le h\le X$) 
$$ 
\Big| \sum_{x< n \le x+ h} \lambda(n) \Big| \ll \frac{h}{(\log h)^{\delta}}  
$$ 
except for at most $X (\log h)^{-\delta}$ integers $x \in [X,2X]$ -- here $\delta$ is a small positive constant, which may be taken as $1/200$ for 
example.  The limit of their technique would be a saving of about $1/\log h$.  In this context, the Riemann hypothesis arguments 
would permit better quantifications: for example, Selberg estimates the quantity in \eqref{1.10} as $O(Xh (\log X)^2)$, and similarly Gao's work shows that the 
variance in \eqref{1.5} is $O(Xh (\log X)^A)$ for a suitable constant $A$.    Thus for a restricted range of $h$, the conditional results exhibit almost a square-root 
  cancelation.   As $h$ tends to infinity, one expects that the sum of the Liouville function in a randomly chosen interval of length $h$ should be distributed approximately like a normal random   variable with mean zero and variance $h$; see \cite{GC}, and \cite{Ng} in the nearly identical context of the M{\" o}bius function, and \cite{MS} for analogous 
  conjectures on primes in short intervals.

 {\bf Acknowledgments.}  I am partly supported through a grant from the National Science Foundation (NSF), and a Simons Investigator 
 grant from the Simons Foundation.    I am grateful to Zeb Brady, Alexandra Florea, Andrew Granville, 
 Emmanuel Kowalski, Robert Lemke Oliver, Kaisa Matom{\" a}ki, Maksym Radziwi{\l \l}, Ho Chung Siu, and Frank Thorne for helpful remarks.  

\section{Preliminaries} 

\subsection{General Plancherel bounds}  

Qualitatively there is no difference between the $L^2$-estimate stated in Theorem \ref{thm1} and 
the $L^1$-estimate 
$$ 
\int_X^{2X} \Big| \sum_{x< n\le x+h} \lambda(n) \Big| dx \le \epsilon X h. 
$$ 
However, the $L^2$ formulation has the advantage that we can use the Plancherel formula to transform the 
problem to understanding Dirichlet polynomials.  We begin by formulating this generally.

\begin{lemm} \label{lem1} 
Let $a(n)$ (for $n=1, 2, 3 \ldots$) denote a sequence of complex numbers  and we suppose that 
$a(n) =0$ for large enough $n$.   Define the associated Dirichlet polynomial 
\begin{equation*} 
\label{2.1} 
A(y) = \sum_{n} a(n) n^{iy}. 
\end{equation*}
Let $T \ge 1$ be a real number.  Then 
$$ 
\int_{0}^{\infty} \Big| \sum_{x e^{-1/T}< n \le x e^{1/T}} a(n) \Big|^2 \frac{dx}{x} =  
\frac{2}{\pi} \int_{-\infty}^{\infty} |A(y)|^2 \Big(\frac{\sin(y/T)}{y}\Big)^2 dy. 
$$ 
\end{lemm}
\begin{proof}  For any real number $x$ put 
$$ 
f(x) = \sum_{e^{x -1/T} \le n \le e^{x+1/T} } a(n), 
$$ 
so that its Fourier transform ${\hat f}(\xi)$ is given by 
$$ 
{\hat f}(\xi) = \int_{-\infty}^{\infty} f(x) e^{-ix\xi} dx = \sum_{n} a(n) \int_{\log n -1/T}^{\log n+1/T} e^{-ix\xi} dx 
= A(-\xi) \Big(\frac{2\sin(\xi/T)}{\xi}\Big). 
$$
The left side of the identity of the lemma is $\int_{-\infty}^{\infty} |f(x)|^2 dx$, and the right side is 
$\frac{1}{2\pi} \int_{-\infty}^{\infty} |{\hat f}(\xi)|^2 d\xi$, so that by Plancherel the stated identity holds. 
\end{proof}

In Lemma \ref{lem1} we have considered the sequence $a(n)$ in ``multiplicatively" short intervals 
$[xe^{-1/T}, xe^{1/T}]$ which is best suited for applying Plancherel, whereas in Theorem \ref{thm1} we 
are interested in ``additively" short intervals $[x,x+h]$.   A simple technical device (introduced by Saffari and Vaughan \cite{SV}) allows 
one to pass from the multiplicative situation to the additive one.  

\begin{lemm} \label{lem2}  Let $X$ be large, and let $a(n)$ and $A(y)$ be as in Lemma \ref{lem1}, and suppose that $a(n) =0$ unless $c_1 X \le n \le c_2 X$ for some 
positive constants $c_1$ and $c_2$.  Let  $h$ be a real number with $1\le h\le c_1X/10$.  Then 
$$ 
\int_0^{\infty} \Big| \sum_{x< n\le x+h} a(n) \Big|^2 dx \ll \frac{c_2^2}{c_1} X \int_{-\infty}^{\infty} |A(y)|^2 \min 
\Big( \frac{h^2}{c_1^2X^2}, \frac 1{y^2}\Big) dy. 
$$ 
\end{lemm}  

\begin{proof}   Temporarily define ${\mathcal A}(x) = \sum_{n\le x} a(n)$.  Note that for any $\nu \in [2h,3h]$ 
$$ 
\int_{0}^{\infty} |{\mathcal A}(x+h) -{\mathcal A}(x)|^2 dx \le 2 \int_0^{\infty} (|{\mathcal A}(x+\nu) - {\mathcal A}(x)|^2 + |{\mathcal A}(x+h)-{\mathcal A}(x+\nu)|^2) dx. 
$$ 
Integrate this over all $2h \le \nu \le 3h$, obtaining that $h$ times the left side above is 
\begin{align}
\label{2.21} 
&\ll \int_{2h}^{3h} \int_0^{\infty} |{\mathcal A}(x+\nu ) -{\mathcal A}(x)|^2 dx \ d\nu + \int_{2h}^{3h} \int_0^{\infty} |{\mathcal A}(x+\nu-h) -{\mathcal A}(x)|^2 dx  \ d\nu \nonumber\\ 
&\ll \int_{c_1X/2}^{c_2X} \int_{h}^{3h} |{\mathcal A}(x+\nu) - {\mathcal A}(x)|^2 d\nu \ dx .  
\end{align}
Now in the inner integral over $\nu$ we substitute $\nu =\delta x$, so that $\delta$ lies between $h/(c_2X)$ and $6 h/(c_1X)$.   
 It follows that the quantity in \eqref{2.21} is 
 \begin{align*} 
& \ll \int_{c_1X/2}^{c_2X}  \int_{h/(c_2X)}^{6h/(c_1X)} |{\mathcal A}(x(1+\delta)) -{\mathcal A}(x)|^2 x d\delta \ dx \\
&=  \int_{h/(c_2X)}^{6h/(c_1X)}  \int_{c_1X/2}^{c_2X} |{\mathcal A}(x(1+\delta)) -{\mathcal A}(x)|^2 x dx \ d\delta \\  
&\ll \frac{c_2^2 hX}{c_1} \max_{h/(c_2 X) \le \delta \le  6h/(c_1 X)} \int_{c_1X/2}^{c_2X} |{\mathcal A}(x(1+\delta)) -{\mathcal A}(x)|^2 \frac{dx}{x},
\end{align*} 
and now, appealing to Lemma \ref{lem1} (with $T=2/\log (1+\delta)$ and noting that $(\sin(y/T)/y)^2 \ll \min(1/T^2,1/y^2)$), the stated result follows.  
\end{proof}

\subsection{The Vinogradov-Korobov zero-free region} 

The Vinogradov-Korobov zero-free region establishes that $\zeta(\sigma+it) \neq 0$ in the 
region 
$$ 
\sigma \ge  1 - C (\log (3+|t|))^{-2/3} (\log \log (3+|t|))^{-1/3},
$$ 
for a suitable positive constant $C$.   Moreover, one can obtain good bounds for $1/\zeta(s)$ in this 
region; see Theorem 8.29 of \cite{IK}.  

\begin{lemm} \label{lem3}  For any $\delta >0$, uniformly in $t$ we have 
$$ 
\sum_{n\le x} \lambda(n) n^{it} \ll x \exp\Big( - \frac{\log  x}{(\log (x+|t|))^{\frac 23 +\delta}} \Big), 
$$ 
and 
$$ 
\sum_{p\le x} p^{it} \ll \frac{\pi(x)}{1+|t|}  + x  \exp\Big( - \frac{\log  x}{(\log (x+|t|))^{\frac 23 +\delta}} \Big).
$$
\end{lemm}  
\begin{proof}  Perron's formula shows that, with $c=1+1/\log x$,  
$$ 
\sum_{n\le x} \lambda(n)n^{it} = \frac{1}{2\pi i} \int_{c-ix}^{c+ix} \frac{\zeta(2w-2it)}{\zeta(w-it)} \frac{x^w}{w} dw + O(x^{\epsilon}).
$$ 
Move the line of integration to Re$(w) = 1- (\log (x+|t|))^{-2/3-\delta}$, staying within the zero-free region for $\zeta(w-it)$.  Using the bounds in Theorem 8.29 of \cite{IK}, the first statement of the lemma follows.  
The second is similar. 
\end{proof} 

For reference, let us note that the Riemann hypothesis gives uniformly 
\begin{equation} 
\label{2.2} 
\sum_{p\le x} p^{it} \ll \frac{\pi(x)}{1+|t|}  + x^{1/2} \log (x+|t|).  
\end{equation}

\subsection{Mean values of Dirichlet polynomials} 

\begin{lemm} \label{lem4}  For any complex numbers $a(n)$ we have 
$$ 
\int_{-T}^{T} \Big| \sum_{n\le N} a(n) n^{it} \Big|^2 dt \ll (T+N) \sum_{n\le N} |a(n)|^2. 
$$ 
\end{lemm} 
\begin{proof}  This mean value theorem for Dirichlet polynomials can be readily 
derived from the Plancherel bound Lemma \ref{lem1}, or see Theorem 9.1 of \cite{IK}. 
\end{proof} 

We shall draw upon Lemma \ref{lem4} many times; one important way in which it 
is useful is to bound the measure of the set on which a Dirichlet polynomial over the primes can be 
large.  

\begin{lemm} \label{lem5}  Let $T$ be large, and $2\le P\le T$.  
Let $a(p)$ be any sequence of complex numbers, defined on primes 
$p$, with $|a(p)| \le 1$.  Let $V \ge 3$ be a real number and let ${\mathcal E}$ denote the set of values $|t|\le T$ such that 
$|\sum_{p\le P} a(p) p^{it}| \ge \pi(P)/V$.  Then  
$$ 
|{\mathcal E}| \ll  (V^2 \log T)^{1+ (\log T)/(\log P)}. 
$$ 
\end{lemm}  
\begin{proof} Let $k= \lceil (\log T)/(\log P) \rceil$ so that $P^k \ge T$.   Write 
$$ 
\Big( \sum_{p\le P} a(p)p^{it}\Big)^k = \sum_{n\le P^{k}} a_k(n) n^{it}. 
$$ 
Note that $|a_k(n)| \le k!$ and that 
$$ 
\sum_{n\le P^k} |a_k(n)| \le \Big( \sum_{p\le P} |a(p)| \Big)^k \le \pi(P)^k. 
$$ 
Therefore, using Lemma \ref{lem4}, we obtain 
$$ 
|{\mathcal E}| \Big( \frac{\pi(P)}{V}\Big)^{2k} \le \int_{-T}^{T} \Big| \sum_{p\le P} a(p)p^{it} \Big|^{2k} dt  
\ll (T+P^k) \sum_{n\le P^k} |a_k(n)|^2 \ll  k! P^k\pi(P)^{k}. 
$$ 
The lemma follows   from the prime number theorem and Stirling's formula.  
\end{proof} 

\subsection{The Hal{\' a}sz-Montgomery bound} 

The mean value theorem of Lemma \ref{lem4}  gives a satisfactory bound when averaging 
over all $|t|\le T$.  We shall encounter averages of Dirichlet polynomials restricted to certain 
small exceptional sets of values $t\in [-T,T]$.  In such situations, an idea going back to Hal{\' a}sz and 
Montgomery, developed in connection with zero-density results, is extremely useful (see Theorem 7.8 of 
\cite{Mon}, or Theorem 9.6 of \cite{IK}).

\begin{lemm}\label{lem6} Let $T$ be large, and ${\mathcal E}$ be a measurable subset of $[-T,T]$.  Then  for any 
complex numbers $a(n)$ 
$$ 
\int_{\mathcal E} \Big| \sum_{n\le N} a(n) n^{it} \Big|^2 dt \ll (N+ |{\mathcal E}|T^{\frac 12} \log T ) \sum_{n\le N} |a(n)|^2. 
$$
\end{lemm} 
\begin{proof}  Let $I$ denote the integral to be estimated, and let $A(t) = \sum_{n\le N} a(n) n^{it}$.  Then 
$$ 
I = \int_{\mathcal E} \sum_{n\le N} \overline{a(n)} n^{-it} A(t) dt \le \sum_{n\le N} |a(n)| \Big| \int_{\mathcal E} A(t) n^{-it} dt \Big|. 
$$  
Using Cauchy-Schwarz we obtain 
\begin{equation} 
\label{2.61} 
I^2 \le \Big( \sum_{n\le N} |a(n)|^2 \Big) \Big( \sum_{n \le 2N} \Big(2- \frac{n}{N} \Big)  \Big| \int_{\mathcal E} A(t) n^{-it} dt \Big|^2 \Big), 
\end{equation}  
where we have taken advantage of positivity to smooth the sum over $n$ in the second sum a little.  Expanding out the integral, the second 
term in \eqref{2.61} is bounded by 
\begin{equation} 
\label{2.62} 
\int_{t_1, t_2 \in {\mathcal E}} A(t_1) \overline{A(t_2)} \sum_{n\le 2N} \Big( 2 - \frac{n}{N} \Big) n^{i(t_2 - t_1)} dt_1 \ dt_2.
\end{equation} 
Now a simple argument (akin to the P{\' o}lya-Vinogradov inequality)  shows that 
\begin{equation} 
\label{2.63} 
\sum_{n\le 2N} \Big(2-\frac{n}{N} \Big) n^{it} \ll \frac{N}{1+|t|^2} + (1+|t|)^{1/2} \log (2+|t|);  
\end{equation} 
here the smoothing in $n$ allows us to save $1+|t|^2$ in the first term, while the unsmoothed sum would have $N/(1+|t|)$ instead (see the proof of Theorem 7.8 in \cite{Mon}).   Using this, and bounding $|A(t_1)A(t_2)|$ by $|A(t_1)|^2 + |A(t_2)|^2$,  we 
see that the second term in \eqref{2.61} is 
$$ 
\ll \int_{t_1 \in {\mathcal E}} |A(t_1)|^2 \Big( \int_{t_2 \in {\mathcal E} } \Big( \frac{N}{1+|t_1-t_2|^2} + T^{1/2 } \log T \Big) dt_2 \Big) dt_1 
\ll  \big( N + |{\mathcal E}| T^{1/2} \log T \big) I. 
$$ 
Inserting this in \eqref{2.61}, the lemma follows.  
\end{proof}

\section{A first attack on Theorem 1.1}
 
 In this section we  establish Theorem \ref{thm1} in the restricted 
 range $h \ge \exp((\log X)^{3/4})$.   This already includes the range $h>X^{\epsilon}$ 
 for any $\epsilon> 0$, and moreover the proof is simple, depending only on Lemmas 
 \ref{lem2}, \ref{lem3} and \ref{lem4}.   Since a large interval may be broken down into 
 several smaller intervals, we may assume that $h\le \sqrt{X}$.  
 
 
 Let ${\mathcal P}$ denote the set of primes in the interval from $\exp((\log h)^{9/10})$ to $h$.  Let us further 
 partition the primes in ${\mathcal P}$ into dyadic intervals.  Thus, let ${\mathcal P}_{j}$ 
 denote the primes in ${\mathcal P}$ lying  between $P_{j}= 2^j \exp((\log h)^{9/10})$ and $P_{j+1} =2^{j+1} \exp((\log h)^{9/10})$.  
 Here $j$ runs from $0$ to $J=\lfloor  (\log h - (\log h)^{9/10})/\log 2\rfloor$.  This choice of ${\mathcal P}$ was made with 
 two requirements in mind: all elements in ${\mathcal P}$ are below $h$, and all are larger than $\exp((\log h)^{9/10})$ 
 which is larger than $\exp((\log X)^{27/40})$, and note that $27/40$ is a little larger than $2/3$ (anticipating an application of 
 Lemma \ref{lem3}).  
 

 Now define a sequence $a(n)$ by setting 
 \begin{equation} 
 \label{3.1} 
 A(y) = \sum_{n} a(n) n^{iy} = \sum_{j} \sum_{p \in {\mathcal P}_{j}} \sum_{X/P_{j+1} \le m \le 2X/P_j } \lambda(p) p^{iy} \lambda(m) m^{iy}.  
 \end{equation}  
 In other words, $a(n) = 0$ unless $X/2 \le n \le 4X$, and  in the range $X\le n \le 2X$ we have 
 \begin{equation}  
 \label{3.2} 
 a(n) = \lambda(n) \omega_{\mathcal P}(n), \qquad \text{where} \qquad \omega_{\mathcal P}(n) =  \sum_{\substack{p \in {\mathcal P} \\ p|n }} 1. 
 \end{equation}

  Tur{\' a}n's proof of the Hardy-Ramanujan theorem can easily be adapted to show that 
 for $n\in [X,2X]$ the quantity $\omega_{\mathcal P}(n)$ is usually close to its average 
 which is about $W({\mathcal P}) = \sum_{p\in {\mathcal P}} 1/p \sim (1/10) \log \log h$.   Precisely, 
 \begin{equation} 
 \label{3.3} 
 \sum_{X\le n\le 2X} \big(\omega_{\mathcal P}(n) - W({\mathcal P}) \big)^2 \ll X W({\mathcal P}) \ll X \log \log h. 
 \end{equation} 
 Moreover, note that for all $X/2 \le n \le 4X$ one has $|a(n)| \le \omega_{\mathcal P}(n)$ and so 
 \begin{equation} 
 \label{3.4} 
 \sum_{n} |a(n)|^2 \ll X W({\mathcal P})^2. 
 \end{equation}

Now 
\begin{align*} 
W({\mathcal P})^2  \int_X^{2X} \Big( \sum_{x< n\le x+h} \lambda(n)\Big)^2 dx 
&\ll \int_{X}^{2X}  \Big( \sum_{x< n< x+h} \lambda(n) \omega_{\mathcal P}(n) \Big)^2  dx \\ 
&+ 
\int_{X}^{2X} \Big(\sum_{x<n \le x+h}  \lambda(n) (\omega_{\mathcal P}(n) - W({\mathcal P}) ) \Big)^2  dx,    
\end{align*}
and the Cauchy-Schwarz inequality and \eqref{3.3} show that the second term above is 
$$ 
\ll \int_X^{2X} h \sum_{x<n \le x+h} (\omega_{\mathcal P}(n) - W({\mathcal P}))^2 dx \ll Xh^2 W({\mathcal P} ).
$$ 
Combining this with Lemma \ref{lem2} we conclude that 
\begin{equation} 
\label{3.5} 
\int_X^{2X} \Big( \sum_{x< n\le x+h} \lambda(n)\Big)^2 dx 
\ll \frac{X}{W({\mathcal P})^2} \int_{-\infty}^{\infty} |A(y)|^2 \min \Big( \frac{h^2}{X^2}, 
\frac{1}{y^2} \Big) dy + \frac{Xh^2}{W({\mathcal P})}. 
\end{equation}

It remains now to estimate the integral in \eqref{3.5}.  First we dispense with the $|y| \ge X$ contribution to the integral, which will be negligible.  Indeed 
by splitting into dyadic ranges $2^k X \le |y| \le 2^{k+1}X$ and using Lemma \ref{lem4} we obtain 
\begin{equation} 
\label{3.6} 
\int_{|y|> X }|A(y)|^2 \frac{dy}{y^2} \ll \frac{1}{X} \sum_{X/2 \le n \le 4X} a(n)^2 \ll W({\mathcal P})^2. 
\end{equation}

Now consider the range $|y|\le X$.  From the definition \eqref{3.1} and Cauchy-Schwarz  
we see that (note $\lambda(p) =-1$) 
\begin{align*}
|A(y)|^2 &\le \Big( \sum_{j=0}^{J} \frac{1}{\log P_j} \Big) \Big(  \sum_{j=0}^{J} \log P_j  \Big| \sum_{p\in {\mathcal P}_j} p^{iy} \Big|^2 \Big| 
\sum_{X/P_{j+1} \le m \le 2 X/P_{j} } \lambda(m) m^{iy} \Big|^2\Big). \\
\end{align*}  
Thus, setting 
\begin{equation}
\label{3.65}  
I_j = (\log P_j)^2 \int_{-X}^{X} \Big| \sum_{p\in {\mathcal P}_j} p^{iy} \Big|^2 \Big |\sum_{X/P_{j+1} \le m \le 2X/P_j} \lambda(m)m^{iy} \Big|^2 \min \Big( \frac{h^2}{X^2}, \frac{1}{y^2}\Big) dy, 
\end{equation}
and noting that $\sum_j 1/ \log P_j \ll W({\mathcal P})$, we obtain 
 %
\begin{align} 
\label{3.7} 
\int_{-X}^{X} |A(y)|^2 \min \Big( \frac{h^2}{X^2}, 
\frac{1}{y^2} \Big) dy 
 \ll W({\mathcal P}) \sum_{j=0}^{J} \frac{1}{\log P_j} I_j \ll W({\mathcal P})^2  \max_{0\le j\le J} I_j. 
\end{align}

To estimate $I_j$, we now invoke Lemma \ref{lem3}.   As noted earlier, our assumption that $h\ge \exp((\log X)^{3/4})$ gives 
 $\log P_j \ge (\log h)^{9/10} \ge (\log X)^{{27}/{40}}$. 
 Thus for $|y|\le X$, Lemma \ref{lem3} shows that 
 $$ 
\sum_{p \in {\mathcal P}_j} p^{iy} \ll \frac{P_j}{\log P_j} \frac{1}{1+|y|} + P_j \exp\big( - (\log X)^{\frac{27}{40}-\frac 23-\delta}\big) \ll 
\frac{P_j}{\log P_j} \Big( \frac{1}{1+|y|} + \frac{1}{\log P_j} \Big),   
$$  
say.  Using this bound for $X\ge |y| \ge \log P_j$, we see that this portion of the integral contributes to $I_j$ an amount 
$$
\ll \frac{P_j^2}{(\log P_j)^2} \int_{\log P_j \le |y|\le X} \Big| \sum_{X/P_{j+1} \le m \le 2X/P_j} \lambda(m) m^{iy}  \Big|^2 \min \Big( \frac{h^2}{X^2}, \frac{1}{y^2} \Big) dy . 
 $$ 
 Split the integral into ranges $|y| \le X/h$, and $2^k X/h \le |y| \le 2^{k+1} X/h$ (for $k=0$, $\ldots$, $\lfloor (\log h)/\log 2\rfloor$) and use Lemma \ref{lem4}.  
Since $X/P_j \gg X/h$, this shows that the quantity above is 
\begin{equation} 
\label{3.8} 
\ll \frac{P_j^2}{(\log P_j)^2} \Big( \frac{h^2}{X^2}  \frac{X^2}{P_j^2}  +\sum_k \frac{h^2}{2^{2k} X^2} \Big( 2^k \frac{X}{h} + \frac{X}{P_j}\Big) \frac{X}{P_j} \Big) \ll  \frac{h^2}{(\log P_j)^2}. 
\end{equation} 
Finally, if $|y| \le \log P_j$, then Lemma \ref{lem3} gives 
\begin{equation} 
\label{3.9} 
\sum_{X/P_{j+1} \le m \le 2X/P_j} \lambda(m) m^{iy} \ll \frac{X}{P_j} (\log X)^{-10},
\end{equation} 
say, so that bounding $\sum_{p\in {\mathcal P}_j} p^{iy}$ trivially by $\ll P_j/(\log P_j)$ we see that this portion of the 
integral contributes to $I_j$ an amount 
\begin{equation} 
\label{3.10}  
\ll (\log P_j)^2 \frac{P_j^2}{(\log P_j)^2} \frac{X^2}{P_j^2} (\log X)^{-20} \frac{h^2}{X^2} (\log P_j) \ll h^2 (\log X)^{-19}. 
\end{equation} 
Combining this with \eqref{3.8}, we obtain that $I_j \ll h^2/(\log P_j)^2$, and so from \eqref{3.7} it 
follows that 
$$ 
\int_{-X}^{X} |A(y)|^2 \min \Big( \frac{h^2}{X^2}, \frac{1}{y^2} \Big) dy \ll W({\mathcal P})^2 \max_j \frac{h^2}{(\log P_j)^2} \ll W({\mathcal P})^2 \frac{h^2}{(\log h)^{9/5}}.
$$ 
Using this and \eqref{3.6} in \eqref{3.5}, we conclude that 
$$ 
\int_X^{2X} \Big( \sum_{x< n\le x+h} \lambda(n)\Big)^2 dx  \ll Xh^2 \Big( \frac{1}{(\log h)^{9/5}} + \frac{1}{W({\mathcal P})} +\frac{1}{h^2}\Big) \ll \frac{Xh^2}{\log \log h}. 
$$ 
This proves Theorem \ref{thm1} in the range $h\ge \exp((\log X)^{ 3/4})$.   

There are two limitations in this argument.  In order to use the mean value theorem (Lemma \ref{lem4}) effectively 
we need to restrict the primes in ${\mathcal P}$ to lie below $h$, so that the Dirichlet polynomial over $m$ has length at least $X/h$.    Secondly, in order to apply Lemma \ref{lem3} to bound the sum over $p\in{\mathcal P}_j$, we are forced to have $P_j > \exp((\log X)^{2/3+ \delta})$ and this motivated our choice of ${\mathcal P}$.    
If we appealed to the Riemann Hypothesis bound \eqref{2.2} instead of Lemma \ref{lem3},  then the second limitation can be relaxed, and the argument presented above would establish Theorem \ref{thm1} in the wider range $h\ge \exp(10 (\log \log X)^{10/9})$.  
In the next section, we shall obtain such  a range unconditionally.  
 
\section{Theorem 1.1 -- Round two}  

We now refine the argument of the previous section, adding another ingredient which will permit 
us to obtain Theorem \ref{thm1}  in the substantially wider region $h \ge \exp( 10(\log \log X)^{10/9})$. 
 Now we shall also need Lemmas \ref{lem5} and \ref{lem6}. 

Let us suppose that $h\le \exp((\log X)^{3/4})$, and let ${\mathcal P}$ and ${\mathcal P}_j$ be as 
in the previous section.  Now we introduce a set of large primes ${\mathcal Q}$ consisting of 
the primes in the interval from $\exp((\log X)^{4/5})$ to $\exp((\log X)^{9/10})$.   As with ${\mathcal P}$, let us also decompose ${\mathcal Q}$ into 
dyadic intervals with ${\mathcal Q}_k$ denoting the primes in ${\mathcal Q}$ lying between $Q_k=2^k \exp((\log X)^{4/5})$ and $Q_{k+1} = 2^{k+1} 
\exp((\log X)^{4/5})$, where $k$ runs from $0$ to $K \sim (\log X)^{9/10}/\log 2$.   

In place of \eqref{3.1} we now define the sequence $a(n)$ by setting 
\begin{equation} 
\label{4.1} 
A(y) = \sum_{n} a(n) n^{iy} = \sum_{j} \Big( \sum_{p \in {\mathcal P}_j} \lambda(p)p^{iy} \Big) A_j(y), 
\end{equation} 
where 
\begin{equation} 
\label{4.11} 
A_j(y) =  \sum_{k} \sum_{\substack{  q \in {\mathcal Q}_k} }   \sum_{X/(P_{j+1}Q_{k+1}) \le m \le 2X/(P_jQ_k)}  
\lambda(q)q^{iy} \lambda(m) m^{iy} . 
\end{equation}
Now $a(n)=0$ unless $X/4\le n \le 8X$, and in the range $X\le n\le 2X$ we have $a(n) = \lambda(n) \omega_{\mathcal P}(n) \omega_{\mathcal Q}(n)$, 
where $\omega_{\mathcal P}(n)$ is as before, and $\omega_{\mathcal Q}$ analogously counts the number of prime factors of $n$ in ${\mathcal Q}$.  

As noted already in \eqref{3.3}, a typical number in $X$ to $2X$ will have $\omega_{\mathcal P}(n) \sim W({\mathcal P})$, and 
similarly will have $\omega_{\mathcal Q}(n) \sim W({\mathcal Q}) = \sum_{q\in {\mathcal Q}} 1/q \sim (1/10) \log \log X$.  Precisely, we 
have 
$$ 
\sum_{X \le n\le 2X} \big(\omega_{\mathcal P}(n) \omega_{\mathcal Q}(n) - W({\mathcal P}) W({\mathcal Q} ) \big)^2 \ll X W({\mathcal P})^2 W({\mathcal Q})^2 \Big( \frac{1}{W({\mathcal P})} + \frac{1}{W({\mathcal Q})}\Big). 
$$ 

Now set (analogously to \eqref{3.65}) 
\begin{equation} 
\label{4.2} 
I_j = (\log P_j)^2  \int_{-X}^{X} \Big| \sum_{p\in {\mathcal P}_j} p^{iy} \Big|^2  |A_j(y)|^2 \min \Big( \frac{h^2}{X^2}, \frac{1}{y^2} \Big) dy. 
\end{equation} 
Then arguing exactly as in \eqref{3.5}, \eqref{3.6}, and \eqref{3.7}, we find that 
\begin{equation} 
\label{4.3} 
\int_X^{2X} \Big( \sum_{x< n\le x+h} \lambda(n)\Big)^2 dx \ll \frac{X}{W({\mathcal Q})^2} \max_{j} I_{j} + \frac{Xh^2}{\log \log h}, 
\end{equation} 
so that our problem has now boiled down to finding a non-trivial estimate for $I_j$.

In the range $|y| \le \log P_j$, we may use a modified version of the bounds in \eqref{3.9} and \eqref{3.10} to see that 
the contribution of this portion of the integral to $I_{j,k}$ is $\ll h^2 (\log X)^{-19}$, which is negligible.  It remains now to bound the integral in \eqref{4.2} in the range $\log P_j \le |y|\le X$.   

Note that we may not be able to use Lemma \ref{lem3} to bound $\sum_{p \in {\mathcal P}_j} p^{iy}$ 
since the range for $p$ might lie below $\exp((\log X)^{ 2/3 +\delta})$.  Define 
\begin{equation} 
\label{4.4} 
{\mathcal E}_j = \Big\{ y: \ \  \log P_j \le |y| \le X,  \ \ \ \Big| \sum_{p \in {\mathcal P}_j} p^{iy} \Big| \ge \frac{P_j}{(\log P_j)^2} \Big \}, 
\end{equation}  
which denotes the exceptional set on which the sum over $p \in {\mathcal P}_j$ does not exhibit much cancelation.  
 To bound the integral in \eqref{4.2} in the range $\log P_j \le |y| \le X$, let us distinguish the cases when $y$ belongs 
to the exceptional set ${\mathcal E}_j$, and when it does not.  Consider the latter case first, where by the definition of ${\mathcal E}_j$ 
the sum over $p \in {\mathcal P}_j$ does have some cancelation.  So this case contributes to \eqref{4.2} 
$$ 
\ll \frac{P_j^2}{(\log P_j)^2}  \int_{-X}^{X} |A_j(y)|^2 \min \Big( \frac{h^2}{X^2}, \frac{1}{y^2} \Big) dy. 
$$
The integral above is the mean value of a Dirichlet polynomial of size about $X/P_j$, which is larger than $X/h$.  Therefore 
applying Lemma \ref{lem4} (as in our estimate \eqref{3.8}) we obtain that the above is 
$$ 
\ll  \frac{P_j^2}{(\log P_j)^2} 
 \frac{h^2}{X^2} \frac{X}{P_j} \sum_{X/(2P_j) \le  n \le 4X/P_j} \Big(\sum_{\substack{ q |n \\ q\in{\mathcal Q}} } 1 \Big)^2 
\ll \frac{h^2}{(\log P_j)^2} W({\mathcal Q})^2. 
$$ 
Thus the contribution of this case to \eqref{4.3} is small as desired.   

Finally we need to bound the contribution of the exceptional values $y \in {\mathcal E}_j$: upon bounding the sum over $p\in {\mathcal P_j}$ 
trivially, this contribution to $I_j$ is 
\begin{equation} 
\label{4.5} 
\ll P_j^2  \int_{{\mathcal E}_j} |A_j(y)|^2 \min \Big( \frac{h^2}{X^2}, \frac{1}{y^2} \Big) dy \ll P_j^2 
\frac{h^2}{X^2} \int_{{\mathcal E}_j} |A_j(y)|^2 dy.
\end{equation}   
Now recall the definition of $A_j(y)$ in \eqref{4.11}, and use Cauchy-Schwarz on the sum over $k$ (as in \eqref{3.65} or \eqref{4.2}) 
to obtain that the quantity in \eqref{4.5} above is 
\begin{equation} 
\label{4.6} 
\ll P_j^2W({\mathcal Q})^2 \frac{h^2}{X^2} \max_k \ (\log Q_k)^2 \int_{{\mathcal E}_j} \Big| \sum_{q \in {\mathcal Q}_k} q^{iy} \Big|^2 \Big| \sum_{X/(P_{j+1}Q_{k+1}) \le m \le 2X/(P_jQ_k)} \lambda(m) m^{iy}  \Big|^2 dy.
\end{equation} 
Since $\log Q_k \ge (\log X)^{4/5}$ (note $4/5$ is bigger than $2/3+\delta$),  in the range 
$X\ge |y| \ge \log P_j$ we can use Lemma \ref{lem3} to obtain 
\begin{equation} 
\label{4.7} 
\sum_{q \in {\mathcal Q}_k} q^{iy} \ll \frac{\pi(Q_{k+1})}{\log P_j} \ll \frac{1}{\log P_j} \frac{Q_k}{\log Q_k}, 
\end{equation}  
which represents a saving of $1/\log P_j$ over 
the trivial bound $Q_k/\log Q_k$.    Using this in \eqref{4.6} and substituting that back in \eqref{4.5}, we see that the contribution of 
the exceptional $y\in {\mathcal E}_j$ to $I_j$ is 
\begin{equation} 
\label{4.8} 
\ll \frac{P_j^2W({\mathcal Q})^2}{(\log P_j)^2} \frac{h^2}{X^2}  \max_{k} Q_k^2  \int_{{\mathcal E}_j}  \Big| \sum_{X/(P_{j+1}Q_{k+1}) \le m \le 2X/(P_jQ_k)} \lambda(m) m^{iy} \Big|^2 dy.
\end{equation} 

It is at this stage that we invoke Lemmas \ref{lem5} and \ref{lem6}.  We are assuming that $\exp(10 (\log \log X)^{10/9}) \le h \le \exp((\log X)^{3/4})$, so that 
$(\log X)^{7} \le P_j \le X^{\epsilon}$.  Appealing to Lemma \ref{lem5}, it follows that $|{\mathcal E}_j| \ll X^{3/7+\epsilon}$.  Using now the bound of 
Lemma \ref{lem6}, 
we conclude that the quantity in \eqref{4.8} is 
\begin{equation} 
\label{4.9} 
\ll \frac{P_j^2W({\mathcal Q})^2}{(\log P_j)^2} \frac{h^2}{X^2} \max_{k} Q_k^2   \Big( \frac{X}{P_jQ_k} + X^{3/7+\epsilon} X^{1/2+\epsilon} \Big) \frac{X}{P_jQ_k} \ll \frac{h^2W({\mathcal Q})^2}{(\log P_j)^2}.   
\end{equation} 
Inserting these estimates back in \eqref{4.3}, we obtain finally that 
$$ 
\int_X^{2X} \Big( \sum_{x < n \le x+h} \lambda(n) \Big)^2 dx \ll \frac{Xh^2}{\log \log h}, 
$$ 
which establishes Theorem \ref{thm1} in this range of $h$.

The limitation in this argument comes from the last step where in applying the Hal{\' a}sz-Montgomery Lemma \ref{lem6} we need the 
measure of the exceptional set ${\mathcal E}_j $ to be a bit smaller than $X^{1/2}$, and to achieve this we needed $P_j$ to be larger 
than a suitable power of $\log X$.  

\section{Once more unto the breach}  

Now we add one more ingredient to the argument developed in the preceding two sections, and this will permit us 
to obtain Theorem \ref{thm1} in the range $h> \exp( (\log \log \log X)^2)$.   
Moreover once this argument is in place, we hope it will be clear that a more 
elaborate iterative argument should lead to Matom{\" a}ki and Radziwi{\l \l}'s result; we shall briefly sketch their argument, where the 
details are arranged differently, in the next section.  Assume below that $h \le \exp(10 (\log \log X)^{10/9})$.   
 
 Let ${\mathcal P}$ and ${\mathcal Q}$ be as in the previous section.   Let ${\mathcal P}^{(1)}$ denote the 
 set of primes lying between $\exp(\exp(\frac{1}{100} (\log h)^{9/10})) $ and $\exp(\exp(\frac{1}{30} (\log h)^{9/10}))$, so that this set is intermediate between ${\mathcal P}$ and ${\mathcal Q}$.   Again split up 
 ${\mathcal P}^{(1)}$ into dyadic blocks, which we shall index as ${\mathcal P}^{(1)}_{j_1}$.   In place of 
 \eqref{4.1} we now define the sequence $a(n)$ by setting 
 \begin{equation} 
 \label{5.1} 
 A(y) =\sum_{n} a(n) n^{iy} =  \sum_{j} \Big(\sum_{p \in {\mathcal P}_j} \lambda(p) p^{iy}\Big) A_j(y), 
 \end{equation} 
 with 
 \begin{equation} 
 \label{5.2} 
 A_j(y) = \sum_{j_1} \Big( \sum_{p_1 \in {\mathcal P}^{(1)}_{j_1}} \lambda(p_1)p_1^{iy} \Big) A_{j,j_1}(y), 
 \end{equation} 
  where, with $M_{j,j_1,k} = X/(P_{j}P^{(1)}_{j_1} Q_{k})$,   
  \begin{equation} 
  \label{5.3} 
  A_{j,j_1} (y) = \sum_{k} \sum_{q\in {\mathcal Q}_k} \sum_{M_{j,j_1,k}/8  \le m  \le 2M_{j,j_1,k}} \lambda(q) q^{iy} \lambda(m)m^{iy} .
  \end{equation} 
  Thus $a(n)$ is zero unless $n$ lies in $[X/8,16X]$ and on $[X,2X]$ we have $a(n) = \lambda(n) \omega_{\mathcal P}(n) \omega_{{\mathcal P}^{(1)}}(n) 
  \omega_{\mathcal Q}(n)$.

Now arguing as in \eqref{4.2} and \eqref{4.3} we obtain 
\begin{equation} 
\label{5.4} 
\int_X^{2X} \Big( \sum_{x < n\le x+h} \lambda(n) \Big)^2 dx \ll \frac{X}{W({\mathcal Q})^2 W({\mathcal P}^{(1)})^2}  \max_j I_j + \frac{Xh^2}{\log \log h}, 
\end{equation} 
where 
\begin{equation} 
\label{5.5}
I_j = (\log P_j)^2 \int_{-X}^{X} \Big| \sum_{p \in {\mathcal P}_j} p^{iy} \Big|^2 |A_j(y)|^2 \min \Big( \frac{h^2}{X^2}, \frac{1}{y^2} \Big) dy. 
\end{equation} 
As before the small portion of the integral with $|y| \le \log P_j$ can be estimated trivially.  Further 
if the sum over $p\in {\mathcal P}_j$ exhibited some cancelation, then the argument of Section 3 applies 
and produces the desired savings (we also saw this in Section 4 when dealing with $y$ not in the exceptional set ${\mathcal E}_j$).  

So now consider the exceptional set ${\mathcal E}_j$ (exactly as in \eqref{4.4}) consisting of $y$ with $\log P_j \le |y| \le X$ and 
$|\sum_{p\in {\mathcal P}_j} p^{iy} | \ge P_j/(\log P_j)^2$, and we must bound the contribution to $I_j$ from $y\in {\mathcal E}_j$.  
As we remarked at the end of Section 4, in the range of $h$ considered here we are 
not able to guarantee that the measure of ${\mathcal E}_j$ is below $X^{1/2-\delta}$, which would have permitted an 
application of Lemma \ref{lem6} (as in Section 4).   Using a Cauchy-Schwarz argument (similar to the ones leading to \eqref{3.65}, or \eqref{4.2}, or \eqref{5.4}),  
we may bound the contribution to $I_j$ from $y \in {\mathcal E}_j$ by 
\begin{equation} 
\label{5.6}
\ll W({\mathcal P}^{(1)})^2 \max_{j_1} (\log P_{j})^2 (\log P^{(1)}_{j_1})^2 I(j,j_1), 
\end{equation} 
say, with 
\begin{equation} 
\label{5.7}
I(j,j_1) = \int_{{\mathcal E}_j} \Big| \sum_{p\in {\mathcal P}_j} p^{iy} \Big|^2 \Big| \sum_{p_1 \in {\mathcal P}^{(1)}_{j_1}} p_1^{iy} \Big|^2 |A_{j,j_1}(y)|^2 \min \Big(\frac{h^2}{X^2},\frac{1}{y^2} \Big) dy.
\end{equation} 

Now ${\mathcal P}^{(1)}$ is a suitably large interval (the lower end point is larger than $(\log X)^{100}$ say), 
so that one can use Lemma \ref{lem5} to show that the measure 
of the set of $y \in [-X,X]$ with  $|\sum_{p_1 \in {\mathcal P}^{(1)}_{j_1}} p_1^{iy} | \ge (P^{(1)}_{j_1})^{9/10}$ is at most $X^{1/3}$.  
For these exceptionally large values of the sum over $p_1$, we bound the sums over $p \in {\mathcal P}_j$ and $p_1\in {\mathcal P}^{(1)}_{j_1}$ 
trivially and argue as in Section 4, \eqref{4.6}--\eqref{4.9}.  This argument shows that the contribution of large values of the sum over $p_1$ to \eqref{5.7} is acceptably small.  

We finally come to the new argument of this section: namely, in dealing with the portion of 
the integral $I(j,j_1)$ where the sum over $p$ is large (since $y\in {\mathcal E}_j$) but the sum over 
$p_1$ exhibits some cancelation.  Bounding the sum over $p_1$ by $\le (P^{(1)}_{j_1})^{9/10}$, we must handle
\begin{equation} 
\label{5.8} 
(P^{(1)}_{j_1})^{9/5}  \int_{{\mathcal E}_j} \Big| \sum_{p \in {\mathcal P}_j} p^{iy} \Big|^2 | A_{j, j_1} (y)|^2 \min \Big( \frac{h^2}{X^2}, \frac{1}{y^2} \Big) dy. 
\end{equation} 
Above we must estimate the mean square of a Dirichlet polynomial of length about $X/P^{(1)}_{j_1}$; the 
set ${\mathcal E}_j$ may not be small enough to use Lemma \ref{lem6} effectively, and the length of 
the Dirichlet polynomial is small compared to $X/h$, so that there is also some loss in using Lemma \ref{lem4}.  
The way out is to bound \eqref{5.8} by 
\begin{equation} 
\label{5.9} (P^{(1)}_{j_1})^{9/5}  \int_{-X}^{X} \Big| \sum_{p \in {\mathcal P}_j} p^{iy} \Big|^{2+2\ell} \Big(\frac{P_j}{(\log P_j)^2}\Big)^{-2\ell} | A_{j, j_1} (y)|^2 \min \Big( \frac{h^2}{X^2}, \frac{1}{y^2} \Big) dy;  
\end{equation} 
here $\ell$ is any natural number, and the inequality holds because on ${\mathcal E}_j$ the sum over $p\in {\mathcal P}_j$ is $\ge P_j/(\log P_j)^2$ by assumption.  We choose $\ell = \lceil (\log P^{(1)}_{j_1})/\log P_j \rceil$. 
 Now in \eqref{5.9}, we must estimate the mean square of the Dirichlet polynomial $(\sum_{p\in {\mathcal P_j}} p^{iy})^{1+\ell} A_{j,j_1}(y)$, and by our choice for $\ell$ this Dirichlet polynomial has length at least $X$, permitting an efficient use of Lemma \ref{lem4}.   With a little effort, Lemma \ref{lem4} can be used to bound \eqref{5.9} by (we have been a little wasteful in some estimates below) 
\begin{align}
\label{5.10} 
&\ll (P^{(1)}_{j_1})^{9/5} \Big(\frac{P_j}{(\log P_j)^2}\Big)^{-2\ell}  \frac{h^2}{X^2} W({\mathcal Q})^2 (\ell+1)! \Big( \frac{(2P_j)^{\ell}X}{P^{(1)}_{j_1}} \Big)^2 \nonumber \\
&\ll W({\mathcal Q})^2 h^2 (P^{(1)}_{j_1})^{-1/5} (\ell \log P_j)^{4\ell} \ll W({\mathcal Q})^2 h^2 (P^{(1)}_{j_1})^{-1/15},  
\end{align}
 where at the last step we used $\log \log P^{(1)}_{j_1} \le (1/30) \log P_j$.  This contribution to 
 \eqref{5.6} is once again acceptably small (having saved a small power of $P^{(1)}_{j_1}$), and 
 completes the proof of Theorem \ref{thm1} in this range of $h$.  
 
 At this stage, all the ingredients in the proof of Theorem \ref{thm1} are at hand, and one 
 can begin to see an iterative argument that would remove even the very weak hypothesis on $h$ made 
 in this section!

\section{Sketch of Matom{\" a}ki and Radziwi{\l \l}'s argument for Theorem 1.1}

In the previous three sections, we have described some of the key ideas developed in 
\cite{MR2}.   The argument given in \cite{MR2} arranges the details differently, in 
order to achieve quantitatively better results: our version saved a modest $\log \log h$ 
over the trivial bound,  and \cite{MR2} saves a small power of $\log h$.   

Instead of considering $a(n)$ being $\lambda(n)$ weighted by the number of primes in 
various intervals (as in Sections 3, 4, 5), Matom{\" a}ki and Radziwi{\l \l} deal with 
$a(n)$ being $\lambda(n)$ when $n$ is restricted to integers with at least one prime factor in 
carefully chosen intervals (and $a(n)=0$ otherwise).  To illustrate, we revisit the argument in Section 3, 
and let ${\mathcal P}$ be the interval defined there.  Let ${\mathcal S}$ denote the set of integers $n \in [1,2X]$ 
with $n$ having at least one prime factor in ${\mathcal P}$.   A simple sieve argument shows that there are 
$\ll X/(\log h)^{1/10}$ numbers $n\in [X,2X]$ that are not in ${\mathcal S}$.  Therefore 
\begin{equation} 
\label{6.1} 
\int_X^{2X}  \Big( \sum_{x< n\le x+h} \lambda(n) \Big)^2 dx \ll 
\int_X^{2X} \Big( \Big(\sum_{\substack{ x< n\le x+h \\ n\in {\mathcal S}} } \lambda(n) \Big)^2 
+ h \sum_{\substack{x< n\le x+h \\ n\notin{\mathcal S}} } 1 \Big) dx, 
\end{equation} 
and the second term is $O(Xh^2/(\log h)^{1/10})$.  Now we use Lemma \ref{lem2} to transform the 
problem of estimating the first sum above to that of bounding the Dirichlet polynomial 
\begin{equation} 
\label{6.2} 
A(y)  = \sum_{\substack{X <n \le 2X \\ n\in{\mathcal S}} } \lambda(n) n^{iy} . 
\end{equation} 
To proceed further, we need to be able to factor the Dirichlet polynomial $A$: this can 
be done by means of the approximate identity 
\begin{equation} 
\label{6.3} 
A(y) \approx \sum_{p \in {\mathcal P}} \sum_{ \substack{m \\ pm \in [X,2X] \\ pm \in {\mathcal S}}} \frac{\lambda(m) m^{iy} }{\omega_{\mathcal P}(m) +1} \lambda(p)p^{iy} . 
\end{equation} 
(The approximate identity above fails to be exact because $n$ might have repeated prime factors from ${\mathcal P}$, but 
this difference is of no importance.)
Now above we can use a standard Fourier analytic technique to separate the variables $m$ and $p$, and 
in this fashion make $p$ and $m$ range over suitable dyadic intervals.  Alternatively one can divide the 
sum over ${\mathcal P}$ into many short intervals, and for each such short interval the corresponding range 
for $m$ may be well approximated by a suitable interval; this is the approach taken in \cite{MR2}.  In either case, 
we obtain a factorization of $A(y)$ very much like what we had in Section 3, and now the argument can follow 
as before.  Note that in the first step \eqref{6.1} we now have a loss of only $O(Xh^2/(\log h)^{1/10})$ which is 
substantially better than our previous argument in \eqref{3.5} where we had the bigger error term $O(Xh^2/ \log \log h)$.

Jumping to the argument in Section 5, we can take ${\mathcal S}$ to be the set of integers $n \in [1,2X]$ 
having at least one prime factor in  each of the intervals ${\mathcal P}$,  ${\mathcal P}^{(1)}$, and ${\mathcal Q}$.  
Once again the sieve shows that there are $\ll X/(\log h)^{1/10}$ integers in $[X,2X]$ that are not in ${\mathcal S}$.  
We start with the expression \eqref{6.3}, and perform a dyadic decomposition of $p\in {\mathcal P}$.   If for each $j$ 
the sum $\sum_{p \in {\mathcal P}_j} p^{iy}$ exhibits cancelation, then using Lemma \ref{lem4} and \eqref{6.3} we 
obtain a suitable bound.   If on the other hand for some $j$ the sum over $p\in {\mathcal P}_j$ is large, then we decompose 
the corresponding Dirichlet polynomial $A_j(y)$ using the primes in ${\mathcal P}^{(1)}$: 
\begin{align} 
\label{6.4} 
A_j(y) &= \sum_{\substack{ m \in [X/P_{j+1}, X/P_j] \\ m \in {\mathcal S}^{(1)} } } \frac{\lambda(m)m^{iy}}{\omega_{\mathcal P}(m) + 1} 
\nonumber\\
&\approx \sum_{p_1 \in {\mathcal P}^{(1)}} \lambda(p_1)  p_1^{iy} \sum_{\substack{ m \\ mp_1 \in [X/P_{j+1}, X/P_j] \\ mp_1 \in {\mathcal S}^{(1)} } } 
\frac{\lambda(m)m^{iy}}{(\omega_{\mathcal P}(m)+1)( \omega_{{\mathcal P}^{(1)}} (m)+1)},
\end{align} 
where ${\mathcal S}^{(1)}$ denotes the integers in $[1,2X]$ with at least one prime factor in ${\mathcal P}^{(1)}$ and one in ${\mathcal Q}$.   
Once again we can split up the primes in ${\mathcal P}^{(1)}$ into dyadic blocks, and separate variables.  If now the sum over $p_1 \in {\mathcal P}_{j_1}^{(1)}$ 
always has some cancelation, then we can argue using an appropriately large moment of the sum over $p\in {\mathcal P}_j$ as in \eqref{5.8}--\eqref{5.10}.  
If for some $j_1$, the sum over $p_1 \in {\mathcal P}_{j_1}^{(1)}$ is large, then we exploit the fact that this set has 
small measure, and argue as in \eqref{4.5}--\eqref{4.9}.  In short the decompositions \eqref{6.3} and \eqref{6.4} give the 
same flexibility as the factorized expressions \eqref{5.1} and \eqref{5.2} that we used in Section 5.

The argument in \cite{MR2} generalizes the approach described in the previous paragraph.  
Matom{\" a}ki and Radziwi{\l \l} define a sequence of increasing ranges of primes, starting with ${\mathcal P}={\mathcal P}^{(0)}$ (as 
in our exposition), and proceeding with ${\mathcal P}^{(1)}$, $\ldots$, ${\mathcal P}^{(L)}$ with the last interval 
getting up to primes of size $\exp(\sqrt{\log X})$, and a final interval ${\mathcal Q}$ (again as in our exposition).   Then 
one restricts to integers having at least one prime factor in each of these intervals.  The corresponding Dirichlet series 
admits many flexible factorizations as in \eqref{6.3} and \eqref{6.4}.  Start with the decomposition \eqref{6.3}, and split into 
dyadic blocks.  If $y$ is such that for all dyadic blocks ${\mathcal P}_j ={\mathcal P}^{(0)}_j$ one has cancelation in $p^{iy}$, 
then Lemma \ref{lem4} leads to a suitable bound.  Otherwise we proceed to a decomposition as in \eqref{6.4}, and 
see whether for every dyadic block in ${\mathcal P}^{(1)}$ the corresponding sum has cancelation.  If that is the case, then a 
moment argument as in \eqref{5.6}--\eqref{5.10} works.  Else, we must have some dyadic block in ${\mathcal P}^{(1)}$ with 
a large contribution, and we now proceed to a decomposition involving ${\mathcal P}^{(2)}$.   Ultimately we arrive at a dyadic 
interval in ${\mathcal P}^{(L)}$ which makes a large contribution, and now we use that this happens very rarely and argue as in 
\eqref{4.5}--\eqref{4.9}.   The structure of the proof may be likened to a ladder -- a large contribution to a dyadic interval in ${\mathcal P}^{(j)}$ is 
used to force a large contribution to a dyadic interval in ${\mathcal P}^{(j+1)}$ -- and one must choose the intervals ${\mathcal P}^{(j)}$ 
so that the rungs of the ladder are neither too close nor too far apart.  Fortunately the method is robust and 
a wide range of choices for ${\mathcal P}^{(j)}$ work.   We end our sketch of the proof of Theorem \ref{thm1} here, referring to 
\cite{MR2} for further details of the proof, and noting that somewhat related iterated decompositions of Dirichlet polynomials arose 
recently in connection with moments of $L$-functions (see \cite{Har}, \cite{RS}).

\section{Generalizations for multiplicative functions}

As mentioned in \eqref{Mult1}, the work of Matom{\" a}ki and Radziwi{\l \l} establishes 
short interval results for general multiplicative functions $f$ with $-1\le f(n) \le 1$ for all $n$.  
Our treatment so far has been specific to the Liouville function; for example we have freely used the bounds 
of Lemma \ref{lem3} which do not apply in the general situation.  In this section we discuss an important 
special class of multiplicative functions (those that are ``unpretentious"), and give a brief 
indication of the changes to the arguments that are needed.   There is one notable extra ingredient that 
we need -- an analogue of the Hal{\' a}sz-Montgomery Lemma for primes (see Lemma \ref{lem7.1} below).  

 A beautiful 
theorem of Hal{\' a}sz \cite{Hal} (extending earlier work of Wirsing) shows that mean values of bounded complex valued 
multiplicative functions $f$ are small unless $f$ pretends to be the function $n^{it}$ for a suitably small value of $t$.  When the 
multiplicative function is real valued, one can show that the mean value is small unless $f$ pretends to be the function $1$: 
this means that $\sum_{p\le x} (1-f(p))/p$ is small.   There is an extensive literature around Hal{\' a}sz's theorem and its consequences; see for example \cite {GS, HT, Mon2, Ten}.    Let us state one such result precisely:  suppose $f$ is a 
completely multiplicative function taking values in the interval $[-1,1]$, and suppose that 
\begin{equation} 
\label{7.1} 
\sum_{p\le X} \frac{1-f(p)}{p} \ge \delta \log \log X
\end{equation} 
for some positive constant $\delta$.  Then  uniformly for all $|t|\le X$ and all $\sqrt{X} \le x \le X^2$ we 
have 
\begin{equation} 
\label{7.2} 
\sum_{n\le x} f(n) n^{it} \ll \frac{x}{(\log x)^{\delta_1}}, 
\end{equation} 
for a suitable constant $\delta_1$ depending only on $\delta$.

Now let us consider the analogue of Theorem \ref{thm1} for such a completely multiplicative 
function $f$, in the simplest setting of short intervals  of length $\sqrt{X} \ge h \ge \exp((\log X)^{17/18})$ 
(a range similar to that considered in Section 3).   In this range we wish to show that 
\begin{equation} 
\label{7.3} 
\int_X^{2X} \Big( \sum_{x < n\le x+h} f(n) \Big)^2 dx = o (Xh^2), 
\end{equation} 
which establishes \eqref{Mult1} for almost all short intervals in this particular situation.  
 
Let ${\mathcal P}$ denote the primes in $\exp((\log h)^{9/10})$ to $h$, as in Section 3, and 
break it up into dyadic blocks ${\mathcal P}_j$ like before.  Analogously to \eqref{3.1}, we 
define the Dirichlet series 
\begin{equation} 
\label{7.4} 
A(y) = \sum_{n} a(n) n^{iy} = \sum_j \sum_{p\in {\mathcal P}_j} \sum_{X/P_{j+1} \le m \le 2X/P_j}  f(p)p^{iy} f(m)m^{iy}, 
\end{equation} 
so that $a(n)$ is zero unless $X/2\le n\le 4X$ and in the range $X\le n\le 2X$ we have $a(n) = f(n) \omega_{\mathcal P}(n)$; 
all exactly as in \eqref{3.2}.  Now arguing as in \eqref{3.3}--\eqref{3.7} we obtain that 
\begin{equation} 
\label{7.5} 
\int_X^{2X} \Big( \sum_{x< n\le x+h} f(n) \Big)^2 dx \ll X \max_j I_j + \frac{Xh^2}{\log \log h},  
\end{equation} 
where 
\begin{equation} 
\label{7.6} 
I_j = (\log P_j)^2 \int_{-X}^{X} \Big| \sum_{p\in {\mathcal P}_j } f(p) p^{iy} \Big|^2 \Big| \sum_{X/P_{j+1} \le m \le 2X/P_j} f(m)m^{iy} \Big|^2 \min 
\Big(\frac{h^2}{X^2},\frac{1}{y^2}\Big) dy. 
\end{equation} 

Since $f$ is essentially arbitrary, we can no longer use Lemma \ref{lem3} to bound the sum over $p$ above.   
The argument splits into two cases depending on whether the sum over $p\in {\mathcal P}_j$ is large or not.  Let 
\begin{equation} 
\label{7.7} 
{\mathcal E}_j = \Big\{ y: \  |y|\le X, \ \ \Big| \sum_{p\in {\mathcal P}_j} f(p)p^{iy} \Big| \ge \frac{P_j}{(\log P_j)^2} \Big\}, 
\end{equation} 
denote the exceptional set on which the sum over $p$ is large.  On the complement of ${\mathcal E}_j$, it is 
simple to estimate the contribution to $I_j$: namely, using Lemma \ref{lem4}, we may bound this contribution by 
$$ 
\ll \frac{P_j^2}{(\log P_j)^2} \int_{-X}^{X} \Big| \sum_{X/P_{j+1} \le m \le 2X/P_j} f(m)m^{iy} \Big|^2 \min 
\Big(\frac{h^2}{X^2},\frac{1}{y^2}\Big) dy \ll \frac{h^2}{(\log P_j)^2}, 
$$ 
which is acceptably small in \eqref{7.5}.

It remains to estimate the contribution to $I_j$ from the exceptional set ${\mathcal E}_j$.  Here we invoke 
the bound \eqref{7.2}, so that the desired contribution is 
\begin{equation} 
\label{7.8} 
\ll \frac{X^2 (\log P_j)^2}{(\log X)^{2\delta_1} P_j^2} \int_{{\mathcal E}_j} \Big| \sum_{p \in {\mathcal P}_j} f(p) p^{iy} \Big|^2 \min\Big(\frac{h^2}{X^2},\frac{1}{y^2} \Big) dy.
\end{equation} 
Since $h \ge \exp((\log X)^{17/18})$ we have $P_j \ge \exp((\log h)^{9/10}) \ge \exp((\log X)^{17/20})$, and an 
application of Lemma \ref{lem5} shows that the measure of ${\mathcal E}_j$ is $\ll \exp((\log X)^{1/6})$.  This 
is extremely small, and it is tempting to use the Hal{\' a}sz-Montgomery Lemma \ref{lem6} to estimate \eqref{7.8}.  
However this gives an estimate too large by a factor of $\log P_j$, since Lemma \ref{lem6} does not take into account that the Dirichlet 
polynomial in \eqref{7.8} is supported only on the primes.   This brings us to the final key ingredient in \cite{MR2} -- 
a version of the Hal{\' a}sz-Montgomery Lemma for prime Dirichlet polynomials.

\begin{lemm} \label{lem7.1}
 Let $T$ be large, and ${\mathcal E}$ be a measurable subset of $[-T,T]$.  Then for any 
complex numbers $x(p)$ and any $\epsilon>0$, 
$$ 
\int_{\mathcal E} \Big| \sum_{p\le P} x(p) p^{it} \Big|^2 dt \ll  \Big(\frac{P}{\log P}  + |{\mathcal E}| P\exp\Big( - \frac{\log P}{(\log (T+P))^{2/3+\epsilon}} \Big) \Big)\sum_{p\le P} |x(p)|^2. 
$$ 
\end{lemm}
\begin{proof}  We follow the strategy of Lemma \ref{lem6}.  Put $P(t) =\sum_{p\le P} x(p)p^{it}$, and let $I$ denote the integral to be estimated.  Then 
using Cauchy-Schwarz as in \eqref{2.61}, we obtain 
\begin{equation} 
\label{7.9} 
I^2 \le \Big( \sum_{p\le P} |x(p)|^2 \Big) \Big( \sum_{p\le 2P} \Big(2-\frac{p}{P}\Big) \Big| \int_{{\mathcal E}} P(t) p^{-it} dt \Big|^2\Big). 
\end{equation} 
Now expanding out the integral above, as in \eqref{2.62}, the second term of \eqref{7.9} is bounded by 
$$ 
\int_{t_1, t_2 \in {\mathcal E}} P(t_1) \overline{P(t_2)} \sum_{p\le 2P} \Big(2- \frac{p}{P}\Big) p^{i(t_2-t_1)} dt_1 dt_2. 
$$ 
Now in place of  \eqref{2.63}, we can argue as in Lemma \ref{lem3} to obtain 
$$ 
  \sum_{p\le 2P} \Big(2- \frac{p}{P}\Big) p^{it} \ll \frac{\pi(P)}{1+|t|^2} + P \exp\Big( -\frac{(\log P)}{(\log (T+P))^{2/3+\epsilon}}\Big), 
  $$ 
 where once again the small smoothing in the sum over $p$ produces the saving of $1+|t|^2$ in the first term.  Inserting this 
 bound in \eqref{7.9}, and proceeding as in the proof of Lemma \ref{lem6}, we readily obtain our lemma. 
\end{proof}

Returning to our proof, applying Lemma \ref{lem7.1} we see that the quantity in \eqref{7.8} may be 
bounded by 
$$ 
\ll \frac{X^2(\log P_j)^2}{(\log X)^{2\delta_1} P_j^2} \Big( \frac{P_j}{\log P_j} + P_j \exp\Big((\log X)^{1/6} - \frac{(\log X)^{17/20}}{(\log X)^{2/3+\epsilon}}\Big) \Big) 
\frac{P_j}{\log P_j} \ll 
 \frac{h^2}{(\log X)^{2\delta_1}}. 
$$ 
 Thus the contribution of $y \in {\mathcal E}_j$ to $I_j$ is also acceptably small, and therefore \eqref{7.3} follows.

\section{Sketch of the corollaries} 

We discuss briefly the proofs of Corollaries \ref{cor2} and \ref{cor3}, starting with Corollary \ref{cor2}.  
The indicator function of smooth numbers is multiplicative, and so Matom{\" a}ki and Radziwi{\l \l}'s general result for multiplicative functions  (see the discussion around \eqref{Mult1}) shows the following:  For any $\epsilon>0$ there exists $H(\epsilon)$ such that  for large enough $N$ 
  the set 
  $$ 
  {\mathcal E} = \{ x \in [\sqrt{N}/2, 2\sqrt{N}]: \  \ \text{ the interval } [x,x+H(\epsilon)] \text{ contains no } N^{\epsilon}\text{-smooth number}\},
  $$ 
  has measure $|{\mathcal E}| \le \epsilon \sqrt{N}$.   Now if for some $x\in [\sqrt{N},2\sqrt{N}]$ we have 
  $x\notin {\mathcal E}$ and also $N/x \notin {\mathcal E}$, then we would be able to find $N^{\epsilon}$-smooth numbers in $[x,x+H(\epsilon)]$ and also in $[N/x,N/x+H(\epsilon)]$ and their product would 
  be in $[N,N+4H(\epsilon)\sqrt{N}]$.   Thus if Corollary \ref{cor2} fails, we must have (with $\chi_{\mathcal E}$ 
  denoting the indicator function of ${\mathcal E}$) 
  $$ 
  \sqrt{N} \le \int_{\sqrt{N}}^{2\sqrt{N}} (\chi_{\mathcal E}(x) + \chi_{\mathcal E}(N/x)) dx \le 4|{\mathcal E}| \le 4\epsilon \sqrt{N},
  $$ 
  which is a contradiction.  

Now let us turn to Corollary \ref{cor3}.   First we recall a beautiful result of Wirsing (see \cite{GS}, or \cite{Ten}),  establishing a conjecture of Erd{\H o}s, which shows that if $f$ is any real valued multiplicative function with $-1\le f(n) \le 1$ then 
$$ 
\lim_{N \to \infty} \frac{1}{N} \sum_{n\le N} f(n)  = \prod_{p} \Big(1-\frac{1}{p}\Big)\Big(1+\frac{f(p)}{p}+\frac{f(p^2)}{p^2} +\ldots \Big). 
$$ 
The product above is zero if $\sum_p (1-f(p))/p$ diverges (this is the difficult part of Wirsing's theorem), 
and is strictly positive otherwise. 

In Corollary \ref{cor3},  we are only interested in the sign of $f$ and so we may 
assume that $f$ only takes the values $0$, $\pm 1$.   Wirsing's theorem applied to $|f|$ shows that condition (ii) of 
the corollary is equivalent to $\sum_{p, f(p)=0} 1/p <\infty$, and further the condition may be restated as  
$$ 
\lim_{N\to \infty} \frac{1}{N} \sum_{n\le N} |f(n)| =\alpha > 0.  
$$ 
Now applying Wirsing's theorem to $f$, it follows that 
$$ 
\lim_{N \to \infty} \frac{1}{N} \sum_{n\le N} f(n) = \beta  
$$ 
exists, and since $f(p) <0$ for some $p$ by condition (i), we also know that $0\le \beta < \alpha$.  
From \eqref{Mult1} we may see that if $h$ is large enough then for all but $\epsilon N$ 
integers $x \in [1,N]$ we must have 
$$ 
\sum_{x<n\le x+h} f(n) \le (\beta+\epsilon)h, \text{  and   } \sum_{x< n \le x+h} |f(n)| \ge (\alpha-\epsilon) h. 
$$ 
Since $\alpha > \beta$, if $\epsilon$ is small enough, this shows that for large enough $h$ (depending on 
$\epsilon$ and $f$) many intervals $[x,x+h]$ contain sign changes of $f$, which gives Corollary \ref{cor3}.










\end{document}